\documentclass{amsart}
\usepackage{amsmath,amssymb,amsthm,bm,bbm}
\usepackage{amscd}
\usepackage{graphicx,enumerate}
\usepackage{layout}
\usepackage{longtable}
\usepackage[all,cmtip]{xy}

\theoremstyle{plain}
\newtheorem{theorem}{Theorem}[section]
\newtheorem*{theorem-nn}{Theorem}
\newtheorem{lemma}[theorem]{Lemma}

\newtheorem*{proposition-nn}{Proposition}

\theoremstyle{definition}
\newtheorem{definition}[theorem]{Definition}

\newtheorem{example}[theorem]{Example}

\newtheorem{remark}[theorem]{Remark}

\theoremstyle{remark}

\newcommand{\bZ}{\mathbbm{Z}}\newcommand{\bQ}{\mathbbm{Q}}

\newcommand{\bC}{\mathbbm{C}}
\newcommand{\bN}{\mathbbm{N}}\newcommand{\bF}{\mathbbm{F}}
\newcommand{\bR}{\mathbbm{R}}

\newcounter{sub}
{\begin{list}{(\arabic{sub})}{\usecounter{sub}%
\setlength{\leftmargin}{2em}}}{\end{list}}

\title{An application of cohomological invariants}

\author[A. Hoshi]{Akinari Hoshi}
\address{Department of Mathematics, Niigata University, Niigata 950-2181,
Japan}
\email{hoshi@math.sc.niigata-u.ac.jp}

\author[M. Kang]{Ming-chang Kang}
\address{Department of Mathematics, National Taiwan University, Taipei, Taiwan}
\email{kang@math.ntu.edu.tw}

\author[A. Yamasaki]{Aiichi Yamasaki}
\address{Department of Mathematics, Kyoto University, Kyoto 606-8502, Japan}
\email{aiichi.yamasaki@gmail.com}

\thanks{{\it Key words and phrases.}
Rationality problem, Noether's problem, cohomological invariants,
unramified cohomology groups.\\
This work was partially supported by JSPS KAKENHI Grant Numbers
 25400027, 16K05059, 19K03418.
}

\subjclass[2010]{Primary 12G05, 14E08.}

\begin{document}

\begin{abstract}
Let $G$ be a finite group, $k$ be a field and
$G\to GL(V_{\rm reg})$ be the regular representation of $G$ over $k$.
Then $G$ acts naturally on the rational function field $k(V_{\rm reg})$
by $k$-automorphisms.
Define $k(G)$ to be the fixed field $k(V_{\rm reg})^G$.
Noether's problem asks whether $k(G)$ is rational
(resp. stably rational) over $k$.

When $k=\bQ$ and $G$ contains a normal subgroup $N$ with $G/H\simeq C_8$ (the cyclic group of order $8$),
Jack Sonn proves that $\bQ(G)$ is not stably rational over $\bQ$,
which is a non-abelian extension of a theorem of Endo-Miyata, Voskresenskii, Lenstra and Saltman for the abelian Noether's problem $\bQ(C_8)$.
Using the method of cohomological invariants, we are able to generalize Sonn's theorem
as follows.
Theorem. Let $G$ be a finite group and $N$ $\lhd$ $G$ such that
$G/N\simeq C_{2^n}$ with $n\geq 3$.
If $k$ is a field satisfying that ${\rm char}\,k=0$ and
$k(\zeta_{2^n})/k$ is not a cyclic extension where $\zeta_{2^n}$ is a primitive $2^n$-th root of unity,
then $k(G)$ is not stably rational (resp. not retract rational) over $k$. Besides, we will also show that $k(G)$ is not stably rational (resp. not retract rational) over $k$ provided that $G$ is a finite group whose $2$-Sylow subgroup is isomorphic to the generalized quaternion group of order $16$ and $k$ is the real quadratic field $\bQ(\sqrt{d})$ where $d$ is a square-free integer of the form $1+8m$ for some integer $m \ge 2$. \end{abstract}

\maketitle


%
\section{Introduction}\label{s1}

Let $G$ be a finite group, $k$ be a field and
$\rho:G\to GL(V_{\rm reg})$ be the regular representation of
$G$ over $k$ where $V_{\rm reg}=\bigoplus_{g\in G}\,k\cdot e_g$
is the regular representation space and $h\cdot e_g=e_{hg}$ for
any $g,h\in G$. Let $k(V_{\rm reg})$ be the function field of $V_{\rm reg}$.
Define $k(G):=k(V_{\rm reg})^G=\{f \in k(V_{\rm reg}): \sigma \cdot f=f$ for any $\sigma \in G \}$, the fixed field of $k(V_{\rm reg})$
under the action of $G$.
Noether's problem asks whether $k(G)$ is stably rational over $k$.

The answer to Noether's problem depends on the group $G$ and on
the field $k$.
For examples, Swan shows that $\bQ(C_{47})$ is not stably rational
over $\bQ$ where $C_n$ denotes the cyclic group of order $n$,
and it is well-known that, if $G$ is a finite abelian group of
exponent $e$ and $k$ is a field containing a primitive $e$-th root
of unity, then $k(G)$ is rational over $k$ (Fischer's Theorem).
On the other hand, Saltman shows that, for any prime number $p$,
there is a non-abelian group $G$ of order $p^9$ such that
$\bC(G)$ is not stably rational over $\bC$ \cite{Sa3}.
We recommend Swan's paper \cite{Sw} for a survey of Noether's problem.

Suppose that $\rho:G\to GL(V)$ is a faithful representation of $G$
over $k$. Then $k(G)$ is stably rational over $k$ if and only if
so is $k(V)^G$ over $k$, by the No-Name Lemma
(see, for examples, \cite[Theorem 5.1]{Ka2}).
Thus some authors would formulate Noether's problem in the following equivalent form : whether $BG$ is
stably rational over $k$ \cite[page 283]{Me} where $BG$ denotes the classifying space of $G$ over $k$; for the definition of the classifying space, see \cite{To}.

\bigskip
Now we turn to the stable rationality of $k(C_{2^n})$.

It is Endo and Miyata, Voskresenskii
who prove that $\bQ(C_{2^n})$ is not stably rational over $\bQ$ if $n\geq 3$
\cite[Corollary 3.1]{EM}. If $k$ is a field satisfying that $k(\zeta_{2^n})/k$
is not a cyclic extension, it is also known that $k(C_{2^n})$ is not stably rational over $k$ (see Lenstra's paper \cite[page 310, Proposition 3.2]{Le} and the proof of the main theorem on page 319 therein). We remark that $k(C_{2^n})$ is retract rational over $k$ if and only if either ${\rm char}\,k=2$ or $k(\zeta_{2^n})/k$
is a cyclic extension (Saltman); see, for examples, \cite[Theorem 3.7]{Ka2}.

\medskip
A new perspective of the above theorem arose from an alternative proof found by Saltman \cite[Theorem 5.11]{Sa1}, who related the negation of the stable rationality of $\bQ(C_{2^n})$ ($n\geq 3$) to the non-existence of a degree $8$ cyclic unramified extension of $\bQ_{2}$ (the $2$-adic local field), i.e. Shianghaw Wang's counter-example to Grunwald's theorem \cite{Wa}.

Following the line of Saltman's idea, Jack Sonn found an unexpected extension to the non-abelian situation.

\bigskip
\begin{theorem}[{Sonn \cite{So}}]\label{t1.1}
Let $G$ be a finite group containing a normal subgroup $N$ such that
$G/N\simeq C_8$.
Then $\bQ(G)$ is not stably rational over $\bQ$.
\end{theorem}

\medskip
The proof of Theorem \ref{t1.1} relies heavily on many delicate arithmetic properties of number fields pertaining to the field $\bQ$.
It is not completely obvious whether Theorem \ref{t1.1}
is valid for a field $k$ other than $\bQ$.

The theory of cohomological invariants is developed by
Serre, Rost, Merkurjev etc. (see \cite{GMS} and \cite{Me1}).
By using Serre's Theorem \cite[page 87, Proposition 33.15]{GMS} and its refined form due to Merkurjev \cite{Me} and V. J. Bailey \cite{Ba} (see Theorem \ref{t2.6}),
we are able to generalize Theorem \ref{t1.1} as follows.

\medskip
\begin{theorem}\label{t1.2}
Let $G$ be a finite group containing a normal subgroup $N$
such that $G/N\simeq C_{2^n}$ $($with $n\geq 3$$)$.
If $k$ is a field of ${\rm char}\,k=0$ satisfying that $k(\zeta_{2^n})/k$
is not a cyclic extension where $\zeta_{2^n}$ is a primitive $2^n$-th
root of unity, then $k(G)$ is not stably rational
$($resp. not retract rational$)$ over $k$.
\end{theorem}

Another situation analogous to Theorem \ref{t1.1} is the following theorem of Serre.

\begin{theorem}[{Serre \cite[page 92, Theorem 34.7]{GMS}}]\label{t1.4}
If $G$ is a finite group whose $2$-Sylow subgroup is isomorphic to $Q_{16}$ $($the generalized quaternion group of order $16$$)$, then $\bQ(G)$ is not stably rational over $\bQ$.
\end{theorem}

Note that the finite groups $G$ satisfying the assumption of the above theorem include $Q_{16}$, $\widetilde{A}_6, \widetilde{A}_7$, $\widehat{S}_4, \widehat{S}_5$ (see \cite[page 90, Example 33.27]{GMS}), and $SL_2(\bF_q)$ where $q \equiv 7$ or $9$ $({\rm mod}\ 16)$(see \cite[page 92]{GMS}). We remark that, if $k$ is a field containing $\bQ$ such that $k(\zeta_8)$ is a cyclic extension of $k$, then  all of $k(Q_{16})$, $k(\widehat{S}_4)$, $k(\widehat{S}_5)$, are rational over $k$ (see \cite{Ka3} and \cite{KZ}).

It seems interesting to find fields $k$ such that Theorem \ref{t1.4} is still valid when we replace $\bQ$ by $k$. Serre indicates one clue to find such a field $k$ : To find a field $k$ and an element $\alpha \in k$ such that $\alpha$ is not a sum of three squares in $k$, but is a sum of three squares in $k(\sqrt{2})$; in the case of $\bQ$, we may choose $\alpha= 7$ (see the proof of \cite[page 85, Proposition 32.25]{GMS}). It is not straightforward to find these fields $k$ except for the easy cases when $k$ is a field extension of $\bQ$ of odd degree or $k$ is a unirational extension of $\bQ$ (see Lemma \ref{l4.2}).

\bigskip
In the following we will show that some real quadratic fields do fulfill the requirement indicated by Serre (and another conditon in Lemma \ref{l4.7}).

\begin{theorem}\label{t1.5}
Let $G$ be a finite group whose $2$-Sylow subgroup is isomorphic to $Q_{16}$ $($the generalized quaternion group of order $16$$)$. Let $k$ be a field containing $\bQ$ such that the quadratic forms $3\langle 1\rangle\bot\langle -7\rangle$ and $8\langle 1\rangle$ are anisotropic over $k$ $($e.g. $k=\bQ$ or $k$ is the real quadratic field $\bQ(\sqrt{d})$ where $d$ is a square-free integer of the form $1+ 8m$ for some integer $m \ge 2$$)$. Then $k(G)$ is not stably rational $($resp. not retract rational$)$ over $k$.
\end{theorem}

\bigskip
The proof of Theorem \ref{t1.2} and Theorem \ref{t1.5} will be given in Section \ref{s3} and Section \ref{s4} respectively. For the convenience of the reader, we include a brief introduction to some basic notions of
cohomological invariants in Section \ref{s2}.

\bigskip
In the remaining part of this section, we digress to explain the notion of retract rationality which appears in Theorem \ref{t1.2} and Theorem \ref{t1.5}.
\begin{definition}\label{d1.3}

According to Saltman \cite[Definition 3.1]{Sa2},
a field extension $L$ of an infinite field $k$ is called {\it retract rational
over $k$} if there are an affine domain $A$ over $k$ and $k$-algebra morphisms
$\varphi:A\to k[X_1,\ldots,X_n][1/f]$,
$\psi:k[X_1,\ldots,X_n][1/f]\to A$ where $k[X_1,\ldots,X_n]$
is the polynomial ring of $k$, $f\in k[X_1,\ldots,X_n]\setminus\{0\}$
satisfying that\\
{\rm (i)} $L$ is the quotient field of $A$, and\\
{\rm (ii)} $\psi\circ\varphi=1_A$, the identity map of $A$.

\medskip
Saltman's notion of retract rationality is generalized by Merkurjev
by waiving the assumption that the base field $k$ is an infinite field
and is formulated as follows (see \cite[page 280, Proposition 3.1]{Me}):
Let $L/k$ be a field extension.
$L$ is called {\it retract rational over $k$} if there are irreducible
quasi-projective varieties $U$, $Y$, $X$ defined over $k$
and $k$-morphisms $\varphi:U\to Y$,
$\psi:Y\to X$ of algebraic varieties such that\\
{\rm (i)} $L$ is the function field of $X$,\\
{\rm (ii)} $U$ is an open subset of $X$, and $\psi\circ\varphi:U\to X$
is the inclusion map of $U$ into $X$, and\\
{\rm (iii)} $Y$ is an open subset of $\mathbbm{A}^n_k$,
the affine space of dimension $n$ over $k$ for some positive integer $n$.

\medskip
To distinguish the above two definitions, we will called them
(S) retract rational and (M) retract rational where
(S) and (M) stand for Saltman and Merkurjev respectively.
If the base field $k$ is an infinite field, then it is obvious that
``(S) retract rational" $\Rightarrow$ ``(M) retract rational''.
In this note, we will focus on (M) retract rational unless otherwise specified; thus
the prefix (M) is omitted henceforth.

\medskip
It is easy to verify that ``rational'' $\Rightarrow$ ``stably rational'' $\Rightarrow$ ``retract rational'' $\Rightarrow$ ``unirational''.

In the lectures of Serre \cite[page 86]{GMS}, two notions
${\rm Noe}(G/k)$ and ${\rm Rat}(G/k)$ are introduced.

Note that ``${\rm Noe}(G/k)$ is true'' $\Rightarrow$
``$k(G)$ is stably rational'' $\Rightarrow$ ``${\rm Rat}(G/k)$ is true''
$\Rightarrow$ ``$k(G)$ is retract rational''
(see \cite[page 25, the third paragraph]{Ka2}).
\end{definition}

\bigskip
Standing notation.
Throughout this note, we consider finite groups $G$.
If $k$ is a field, we regard $G$ as the constant group scheme over $k$
associated to $G$ \cite[page 332]{KMRT}; thus all the results of cohomological invariants
developed in \cite{GMS}, \cite{Me}, \cite{Ba} may be applied to the
situation where $G$ is a finite group.
We denote by $C_n$ the cyclic group of order $n$.
If $k$ is a field with $\gcd\{n,{\rm char}\,k\}=1$, then
$\zeta_n$ denotes a primitive $n$-th root of unity in some extension
field of $k$.
For emphasis, recall that $k(G)$ is the fixed field
$k(V_{\rm reg})^G=k(x_g:g\in G)^G$ where $h\cdot x_g=x_{hg}$
for all $g,h\in G$.
Remember that the classifying space $BG$ over $k$
is stably rational $($resp. retract rational$)$ over $k$ if and only if $k(G)$ is stably rational
(resp. retract rational) over $k$;
see \cite[page 283]{Me}, \cite[Section 2.2]{Ba}.

\bigskip
Acknowledgments. We thank A. S. Merkurjev who pointed out an erroneous use of Totaro's Theorem \cite[page 99]{GMS} in an early version of this article.

\bigskip
\section{Cohomological invariants}\label{s2}

For a general definition of invariants (the invariant of a smooth algebraic group over a field $k$ in a cycle module), the reader may look into \cite[Section 2]{Me1} or \cite[page 7]{GMS}. In this note, we are concerned only with the cohomological invariants, and the smooth algebraic groups we consider are restricted to finite groups.

\medskip
\begin{definition}\label{d2.1}
Let $k$ be the base field and $G$ be a finite group over $k$, i.e. the finite constant group schemes over the field $k$ \cite[page 332]{KMRT}.

Denote by ${\rm Fields}_{/k}$ the category of field extensions $K$
of $k$, by {\rm Sets}  the category of $($pointed$)$ sets,
by {\rm Abelian Groups} the category of abelian groups.
Let $A$ and $H$ be two functors defined by
\begin{align*}
A:{\rm Fields}_{/k}&\longrightarrow{\rm Sets}\\
K&\longmapsto H^1(K,G):=H^1(\Gamma_K,G)
\end{align*}
where $\Gamma_K={\rm Gal}(K_{\rm sep}/K)$ the absolute Galois group of $K$,
\begin{align*}
H:{\rm Fields}_{/k}&\longrightarrow{\rm Abelian\ Groups}\\
K&\longmapsto H^d(K,\bQ/\bZ(d-1)):=H^d(\Gamma_K,\bQ/\bZ(d-1))
\end{align*}
where $d$ is a positive integer and the Galois cohomology
$H^d(K,\bQ/\bZ(d-1))$ is defined as in
\cite[page 278, Section 2.1]{Me}, \cite[pages 7--8, Section 2.3]{Ba}.

\medskip
An {\it $H$-invariant of $A$}
(or a {\it cohomological invariant of $A$}, in short)
is a natural transformation $a:A\to H$ of the functor
$A$ to the functor $H$.

Explicitly, for any $K\in {\rm Fields}_{/k}$, there is a map
\begin{align*}
a_K:H^1(K,G)\to H^d(K,\bQ/\bZ(d-1))
\end{align*}
satisfying the conditions: If $K\subset K^\prime$ are two objects
in ${\rm Fields}_{/k}$, then the following diagram commutes:
\begin{align*}
\begin{CD}
H^1(K,G) @>a_K>> H^d(K,\bQ/\bZ(d-1))\\
@VVV @VVV\\
H^1(K^\prime,G) @>a_{K^\prime}>> H^d(K^\prime,\bQ/\bZ(d-1)).\\\\
\end{CD}
\end{align*}

\medskip
Since $G$ is the constant group scheme, it follows that $\Gamma_K$ acts trivially on $G(K_{\rm sep}) \simeq G$. Thus we may regard the set $H^1(K, G)$ as the set of all group homomorphisms from $\Gamma_K$ to the finite group $G$. On the other hand, the set $H^1(K, G)$ is naturally bijective to the set of $G$-torsors over ${\rm Spec}(K)$, equivalently, the set of $G$-Galois algebras over $K$ \cite[pages 388--389]{KMRT}. Note that a $G$-Galois algebra $A$ over a field $K$ \cite[page 288]{KMRT} is a $G$-Galois extension $A/K$ of commutative rings in the sense of Galois extensions of commutative rings in \cite[page 80--84]{DI}.

\medskip
In $H^1(k,G)$, there is a distinguished element
$\varphi_0:\Gamma_k\to G$ which maps every element of $\Gamma_k$
to the identity element of $G$.
A cohomological invariant $a$ is called {\it normalized} if
$a_k(\varphi_0)=0$ in $H^d(k,\bQ/\bZ(d-1))$.

A cohomological invariant $a$ is called {\it constant} if there is a $d$-cocycle $\gamma$ for the group $\Gamma_k$ such that for any field extension $f:k\to K$, $a_K$ is the constant map whose image is
$[f^\ast(\gamma)]$ where $[f^\ast(\gamma)]$ is the cohomological class associated to the cocycle $f^\ast(\gamma)$ and $f^\ast:\Gamma_K\to\Gamma_k$ is the map induced by $f:k\to K$ \cite[page 10]{GMS}.

Every invariant $a$ can be written as (constant)+(normalized)
(see \cite[page 11]{GMS}).

\medskip
The set of all the $H$-invariants of $A$ is denoted by
\begin{align*}
{\rm Inv}_k^d(G,\bQ/\bZ(d-1));
\end{align*}
it is an abelian group under the pointwise addition.

The set of all the normalized $H$-invariants of $A$ is denoted by
\begin{align*}
{\rm Inv}_k^d(G,\bQ/\bZ(d-1))_{\rm norm};
\end{align*}
it is a subgroup of ${\rm Inv}_k^d(G,\bQ/\bZ(d-1))$.
\end{definition}

\bigskip
\begin{definition}\label{d2.2}
Assume that $k$ is a field with ${\rm char}\,k =0$.

An invariant $a\in {\rm Inv}_k^d(G,\bQ/\bZ(d-1))$ is called
{\it unramified} if for any field extension $K/k\in {\rm Fields}_{/k}$,
any $\varphi\in H^1(K,G)$, the image $a_K(\varphi)$ belongs to
the kernel of $r_v:H^d(K,\bQ/\bZ(d-1))\to
H^{d-1}(k_v,\bQ/\bZ(d-2))$ for all $v$ where $v$
is a discrete $k$-valuation on $K$ with residue field $k_v$ (see \cite[page 19]{GMS}).

The map $r_v$ is called the {\it residue map}.

If $k$ is a field with ${\rm char}\,k =p > 0$, the definition of a unramified invariant over $k$ may be found in \cite[page 9]{Ba}, \cite[page 152]{GMS}.

\medskip
The subgroup of all the unramified normalized $H$-invariants of $A$ is
denoted as
\begin{align*}
{\rm Inv}_{{\rm nr},k}^d(G,\bQ/\bZ(d-1))_{\rm norm}.
\end{align*}
\end{definition}

\bigskip
\begin{theorem}[{\cite[page 87, Proposition 33.10]{GMS}, \cite[page 289, Corollary 6.6]{Me}}]\label{t2.3}
If $k(G)$ is retract rational over $k$, then
${\rm Inv}_{{\rm nr},k}^d(G,\bQ/\bZ(d-1))_{\rm norm}=0$.
Consequently, if
${\rm Inv}_{{\rm nr},k}^d(G,\bQ/\bZ(d-1))_{\rm norm}\neq 0$, then
$k(G)$ is not stably rational $($resp. not retract rational$)$ over $k$.
\end{theorem}

\medskip
Using Theorem \ref{t2.3}, we may prove that $\bQ(C_8)$ is
not retract rational over $\bQ$ by showing the non-triviality of ${\rm Inv}_{{\rm nr},k}^2(G,\bQ/\bZ(1))_{\rm norm}$ where $G=C_8$; this is carried out in \cite[page 292, Example 8.3]{Me}, \cite[Section 4.1]{Ba} (and also in \cite[pages 87--88, Proposition 33.15]{GMS}). Such proofs are different from those of the original proofs of
Endo-Miyata \cite{EM}, Lenstra \cite{Le}, Saltman \cite{Sa1}.

\medskip
\begin{theorem}\label{t2.4}
{\rm (1) ({\cite[page 290, Proposition 7.1]{Me}, \cite[Chapter 3]{Ba}})}
\begin{align*}
{\rm Inv}_k^1(G,\bQ/\bZ)_{\rm norm}&\simeq H^1(G,\bQ/\bZ),\\
{\rm Inv}_k^2(G,\bQ/\bZ(1))_{\rm norm}&\simeq H^2(G,k^\times).
\end{align*}
{\rm (2) (V. J. Bailey \cite[page 31]{Ba})}
${\rm Inv}_{{\rm nr},k}^1(G,\bQ/\bZ)_{\rm norm}=0$.
\end{theorem}

\bigskip
\begin{definition}[{V. J. Bailey \cite[page 34]{Ba}}]\label{d2.5}
By Theorem \ref{t2.4}, we have
${\rm Inv}_k^1(G,\bQ/\bZ)_{\rm norm}\simeq H^1(G,\bQ/\bZ)\simeq H^2(G,\bZ)$
where the last isomorphism is given by the map
$\delta:H^1(G,\bQ/\bZ)\simeq H^2(G,\bZ)$, the connecting homomorphism associated to the short exact sequence
$0\to \bZ\to \bQ\to \bQ/\bZ\to 0$.
Since we have the cup product
\begin{align*}
\cup:H^2(G,\bZ)\otimes_\bZ H^0(G,k^\times)\to
H^2(G,k^\times)\simeq {\rm Inv}^2_k(G,\bQ/\bZ(1))_{\rm norm}
\end{align*}
which is injective by \cite[page 34, Proposition]{Ba},
we define a {\it degree two decomposable invariant} as an element
in the image of the following map:
\begin{align*}
H^1(G,\bQ/\bZ)\otimes_\bZ k^\times
&\longrightarrow H^2(G,k^\times)\simeq{\rm Inv}^2_k(G,\bQ/\bZ(1))_{\rm norm}\\
\chi\otimes b&\longmapsto \delta(\chi)\cup(b)
\end{align*}
where $(b)\in H^0(G,k^\times)$ is identified with $b\in k^\times$
through the isomorphism $H^0(G,k^\times)\simeq (k^\times)^G=k^\times$.

\medskip
A decomposable invariant is normalized \cite[page 34]{Ba}. The subgroup of decomposable invariants in
${\rm Inv}_k^2(G,\bQ/\bZ(1))_{\rm norm}$ is denoted as
\begin{align*}
{\rm Inv}_k^2(G,\bQ/\bZ(1))_{\rm norm}^{\rm dec}.
\end{align*}

The subgroup of unramified invariants of
${\rm Inv}_k^2(G,\bQ/\bZ(1))_{\rm norm}^{\rm dec}$ is denoted as
\begin{align*}
{\rm Inv}_{{\rm nr},k}^2(G,\bQ/\bZ(1))_{\rm norm}^{\rm dec}.
\end{align*}
\end{definition}

\medskip
\begin{theorem}[{V. J. Bailey}]\label{t2.6}
Let $k$ be a field and $G$ be a finite abelian group.
Let $G_0$ be the $2$-Sylow subgroup of $G$ and write
$G_0\simeq \bZ/2^{d_1}\bZ\oplus\cdots\oplus\bZ/2^{d_t}\bZ$
where $d_1\geq\cdots\geq d_t\geq 0$.\\
{\rm (1)} If ${\rm char}\,k=2$, then
${\rm Inv}_{{\rm nr},k}^2(G,\bQ/\bZ(1))_{\rm norm}^{\rm dec}=0$;\\
{\rm (2)} If ${\rm char}\,k\neq 2$ and $k(\zeta_{2^{d_1}})/k$
is a cyclic extension, then
${\rm Inv}_{{\rm nr},k}^2(G,\bQ/\bZ(1))_{\rm norm}^{\rm dec}=0$;\\
{\rm (3)} If ${\rm char}\,k\neq 2$ and $k(\zeta_{2^{d_i}})/k$ is not a cyclic
extension for $1\leq i\leq e$ $(\leq t)$ while $k(\zeta_{2^{e+1}})/k$ is cyclic,
then
${\rm Inv}_{{\rm nr},k}^2(G,\bQ/\bZ(1))_{\rm norm}^{\rm dec}\simeq(\bZ/2\bZ)^{(e)}$.
\end{theorem}
\begin{proof}
Let $p_1,\ldots,p_m$ be the odd prime divisors of $|G|$ and
$G_i$ be the $p_i$-Sylow subgroup of $G$ for $1\leq i\leq m$.
Then $G\simeq\bigoplus_{0\leq i\leq m} G_i$.
Thus we have
\begin{align*}
{\rm Inv}_{{\rm nr},k}^2(G,\bQ/\bZ(1))_{\rm norm}^{\rm dec}\simeq
\bigoplus_{0\leq i\leq m}
{\rm Inv}_{{\rm nr},k}^2(G_i,\bQ/\bZ(1))_{\rm norm}^{\rm dec}
\end{align*}
by \cite[page 42, Lemma 5]{Ba}.
In other words, to prove Theorem \ref{t2.6} we may assume that
$G$ is an abelian $p$-group.

If ${\rm char}\,k\neq p$, we may apply Bailey's theorem in \cite[page 48]{Ba}.
If ${\rm char}\,k=p$, since $k(G)$ is rational over $k$
by Gasch\"utz-Kuniyoshi's Theorem \cite{Ga}, \cite{Ku},
we may apply Theorem \ref{t2.3} and we find that
${\rm Inv}_{{\rm nr},k}^2(G,\bQ/\bZ(1))_{\rm norm}=0$.
\end{proof}

\bigskip
\section{Proof of Theorem \ref{t1.2}}\label{s3}

\begin{lemma}\label{l3.1}
Let $G$ be a finite group and $G^{ab}=G/[G,G]$ where
$[G,G]$ is the commutator subgroup of $G$.
Then the canonical projection $f: G\to G^{ab}$ induces an injective morphism
from ${\rm Inv}_{{\rm nr},k}^2(G^{ab},\bQ/\bZ(1))_{\rm norm}^{\rm dec}$
to ${\rm Inv}_{{\rm nr},k}^2(G,\bQ/\bZ(1))_{\rm norm}$.
\end{lemma}
\begin{proof}
Step 1. Note that ${\rm Hom}(G^{ab},\bQ/\bZ)\simeq{\rm Hom}(G,\bQ/\bZ)$.
In fact, if $\chi:G\to \bQ/\bZ$ is any group homomorphism, then
${\rm Kernel}(\chi)\supset [G,G]$ because the image of $\chi$ is abelian.
Thus $\chi$ factors through $G^{ab}$, i.e. there is a group
homomorphism $\overline{\chi}:G^{ab}\to \bQ/\bZ$ such that $\overline{\chi} \cdot f=\chi$. Conclusion: $\overline{\chi} \mapsto \overline{\chi} \cdot f$ provides an isomorphism for $H^1(G^{ab},\bQ/\bZ) \to H^1(G,\bQ/\bZ)$.

Thus we have a commutative diagram
\begin{align*}
\begin{CD}
H^1(G^{ab},\bQ/\bZ)\otimes_\bZ k^\times @>\simeq >> {\rm Inv}^2_k(G^{ab},\bQ/\bZ(1))^{\rm dec}_{\rm norm}\\
@V\simeq VV @VV\Phi V\\
H^1(G,\bQ/\bZ)\otimes_\bZ k^\times @>\simeq >> {\rm Inv}^2_k(G,\bQ/\bZ(1))^{\rm dec}_{\rm norm}
\end{CD}
\end{align*}
where the horizontal morphisms and the left vertical morphism are isomorphisms. It follows that the right vertical morphism
$\Phi$ is also an isomorphism. Note that $\Phi$ is defined by $\Phi(\delta(\overline{\chi})\cup(b))=\delta(\overline{\chi}\cdot f)\cup(b)$ for any $\overline{\chi}\in H^1(G^{ab},\bQ/\bZ)$, any $b \in k^\times$.

In the remaining part of the proof, we will show that $\Phi$ sends the unramified invariants to the unramified ones. The main idea is to verify that $\Phi$ is essentially the same as the morphism $f^\ast$ defined in \cite[page 33]{GMS} (see Step 3 in the following).

\medskip
Step 2. Recall the definition of a degree two decomposable invariant.

Let $\chi:G\to \bQ/\bZ$ be a group homomorphism and $b \in k^\times$. Consider the invariant $\delta(\chi)\cup(b)$. Write $H$ for the kernel of $\chi$, and let the image of $\chi$ be $\langle
1/n \rangle \subset \bQ/\bZ$. Choose $\sigma \in G/H$ such that $\chi$ induces an isomorphism of $G/H=\langle \sigma \rangle$ onto $\langle 1/n \rangle$ with $\chi(\sigma)=1/n$.

By \cite[pages 37--40]{Ba} $\delta(\chi)\cup(b)$ is defined as follows : for any field extension $K/k$, for any $G$-Galois algebra $L/K$, $(\delta(\chi)\cup(b))_K(L)$ is the similarity class, in ${\rm Br}(K) \simeq H^2(K, \bQ/\bZ(1))$, of the cyclic algebra $[E/K, \sigma, b]$ where $E=L^H$ (the fixed ring of $L$ under $H$). It is known that $E$ is a $G/H$-Galois algebra over $K$ by the fundamental theorem of Galois theory (see \cite[page 80, page 143]{DI}). For the definition of the cyclic algebra $[E/K, \sigma, b]$, see \cite[page 49, page 71]{Dr} and \cite[page 121]{DI}.

\medskip
Step 3.
There is a morphism $f^\ast : {\rm Inv}^2_k(G^{ab},\bQ/\bZ(1))\to {\rm Inv}^2_k(G,\bQ/\bZ(1))$ associated to the group homomorphism $f: G\to G^{ab}$ (see \cite[page 33]{GMS}).

If $a \in {\rm Inv}^2_k(G^{ab},\bQ/\bZ(1))$, then $f^\ast(a) \in {\rm Inv}^2_k(G,\bQ/\bZ(1))$ is defined as follows : for any field extension $K/k$, for any $G$-Galois algebra $L/K$, $f^\ast(a)_K(L):=a_K(L_0)$ where $L_0=L^{[G,G]}$ (the fixed ring of $L$ under the action of $[G,G]$) by the construction of the $G^{ab}$-torsor $T$ in the proof of \cite[page 33, Proposition 13.1]{GMS}. Note that $L_0$ is a $G^{ab}$-Galois algebra over $K$; thus $a_K(L_0)$ is well-defined. Also note that $f^\ast$ sends the normalized invariants (resp. the unramified invariants) to the normalized invariants (resp. the unramified invariants).

\medskip
Step 4.
Consider the following diagram:

\begin{align*}
\begin{CD}
{\rm Inv}^2_k(G^{ab},\bQ/\bZ(1))^{\rm dec}_{\rm norm} @>>> {\rm Inv}^2_k(G^{ab},\bQ/\bZ(1))_{\rm norm}\\
@V\Phi VV @VV f^\ast V\\
{\rm Inv}^2_k(G,\bQ/\bZ(1))^{\rm dec}_{\rm norm} @>>> {\rm Inv}^2_k(G,\bQ/\bZ(1))_{\rm norm}.
\end{CD}
\end{align*}

We will show that the above diagram is a commutative diagram.

For any field extension $K/k$, for any $G$-Galois algebra $L/K$, we will show that $(\delta(\overline{\chi}\cdot f)\cup(b))_K(L)$ and $f^\ast(\delta(\overline{\chi})\cup(b))_K(L)$ represent similar central simple $K$-algebras.

Write $H_0$ for the kernel of $\overline{\chi}$, and let $H$ be the subgroup of $G$ such that $[G,G] \subset H \subset G$ and $H/[G, G] \simeq H_0$. Then $H$ is the kernel of $\overline{\chi}\cdot f$. Define $E:=L^H$ as in Step 2. Then $(\delta(\overline{\chi}\cdot f)\cup(b))_K(L)$ is the similarity class of the cyclic algebra $[E/K, \sigma, b]$. For the notation of the cyclic algebra, see \cite[page 49, page 70]{Dr}; note that the definition $[E/K, \sigma, b]$ is still valid when the field extension $E/K$ is replaced by a Galois algebra (see \cite{DI}). Remember that $H^2(K, \bQ/\bZ(1))$ is isomorphic to the Brauer group ${\rm Br}(K) = H^2(\Gamma_K, K_{\rm sep}^{\times})$.

On the other hand, by Step 3, $f^\ast(\delta(\overline{\chi})\cup(b))_K(L)= (\delta(\overline{\chi})\cup(b))_K(L_0)$ where $L_0=L^{[G,G]}$. Since $L_0^{H_0}=(L^{[G, G]})^{H_0}=L^H=E$, we find that $(\delta(\overline{\chi})\cup(b))_K(L_0)$ is the similarity class of $[E/K, \sigma, b]$. Hence the result.

\medskip
Step 5.
{}From the above commutative diagram, we find that ${\rm Inv}_{{\rm nr},k}^2(G^{ab},\bQ/\bZ(1))_{\rm norm}^{\rm dec}$ is mapped injectively by $f^\ast$ into ${\rm Inv}_k^2(G,\bQ/\bZ(1))_{\rm norm}$. Thus the image of ${\rm Inv}_{{\rm nr},k}^2(G^{ab},\bQ/\bZ(1))_{\rm norm}^{\rm dec}$ lies in ${\rm Inv}_{{\rm nr},k}^2(G,\bQ/\bZ(1))_{\rm norm}$, because $f^\ast$ preserves the unramified invariants. Thus $\Phi$ maps  ${\rm Inv}_{{\rm nr},k}^2(G^{ab},\bQ/\bZ(1))_{\rm norm}^{\rm dec}$ into ${\rm Inv}_{{\rm nr},k}^2(G,\bQ/\bZ(1))_{\rm norm}^{\rm dec}$.
\end{proof}

\bigskip
Proof of Theorem \ref{t1.2} ------------------\\

Let $N$ $\lhd$ $G$ and $G/N\simeq C_{2^n}$ where $n\geq 3$.
It follows that $N\supset [G,G]$ and there is a surjection
$G^{ab}\to G/N\simeq C_{2^n}$.

Let $G_0$ be the $2$-Sylow subgroup of $G^{ab}$ and write
$G_0\simeq \bZ/2^{d_1}\bZ\oplus\cdots\oplus\bZ/2^{d_t}\bZ$
where $d_1\geq \cdots\geq d_t\geq 1$.
Since $G^{ab}$ has a quotient isomorphic to $C_{2^n}$ with $n\geq 3$,
it follows that $d_1\geq n$.

By the assumption that $k(\zeta_{2^n})/k$ is not cyclic,
it follows that $k(\zeta_{2^{d_1}})/k$ is not cyclic.
By Theorem \ref{t2.6},
${\rm Inv}_{{\rm nr},k}^2(G^{ab},\bQ/\bZ(1))_{\rm norm}^{\rm dec}\neq 0$.
Hence
${\rm Inv}_{{\rm nr},k}^2(G,\bQ/\bZ(1))_{\rm norm}\neq 0$
by Lemma \ref{l3.1}.
By Theorem \ref{t2.3}, we find that $k(G)$ is not retract rational over $k$.
\qed

\medskip
\begin{example}\label{e3.2}
In \cite[page 88, Theorem 33.16]{GMS}, Serre proves the following result:
Let $G$ be a finite group whose $2$-Sylow subgroup $P$ is a cyclic group
of order $\geq 8$.
Then ${\rm Rat}(G/\bQ)$ $($resp. ${\rm Noe}(G/\bQ)$$)$ is false.

By Burnside's normal $p$-complement theorem
(see \cite[page 160, Corollary 5.14]{Is}),
$G=N\rtimes P$ where $N$ $\lhd$ $G$ and $|N|$ is odd.
It follows that $G/N\simeq C_{2^n}$ where $n\geq 3$.
Thus we may apply Theorem \ref{t1.2} to obtain an alternative proof of
Serre's theorem when the field $\bQ$ is replaced by a field $k$ such that
${\rm char}\,k=0$ and $k(\zeta_{2^n})/k$ is not cyclic.

The conclusion, for the split case $G=N\rtimes H$ where $N$ and $H$
are any subgroups, is anticipated by Saltman.
In fact, Saltman proves that, for a finite group $G=N\rtimes H$
(where $N$ $\lhd$ $G$), if $k(G)$ is (S) retract rational over an infinite field $k$,
so is $k(H)$ over $k$ (see \cite[page 265, Theorem 3.1 (b)]{Sa1}).
Note that, in \cite{Sa1}, the result is formulated in terms of generic Galois extensions. However, it is known that a group $G$ has a generic Galois extension over $k$ if and
only if $k(G)$ is (S) retract rational over $k$ by \cite[Theorem 3.12]{Sa2} (also see \cite[Theorem 1.2 and pages 28--29]{Ka2}).
\end{example}

\medskip
\begin{example}\label{e3.3}
We will exhibit a non-split extension $1\to N\to G\to G/N\to 1$.
Define $G=\langle\sigma,\tau:\sigma^{64}=\tau^{32}=1,\sigma^{32}
=\tau^{16},\tau^{-1}\sigma\tau=\sigma^7\rangle$, $N=\langle\sigma\rangle$.
Then $G/N\simeq\langle\tau\rangle\simeq C_{32}$.
This example is taken with
the parameters $(r,s,u,v,t,w)=(3,1,0,1,0,1)$
in \cite[Definition 5.5]{Ka1} and \cite[Lemma 5.6 (ii)]{Ka1}.

By Theorem \ref{t1.2}, if $k$ is a field with ${\rm char}\,k=0$ such that
$k(\zeta_{32})/k$ is not cyclic, then $k(G)$ is not retract rational over $k$.
However, if $k':=k(\zeta_{64})$, then  $k'(G)$ is rational over
$k'$ by \cite[Theorem 1.5]{Ka1}.
\end{example}

\bigskip
\section{Proof of Theorem \ref{t1.5}}\label{s4}

We recall some terminology and basic facts of (non-degenerate) quadratic forms first.

Throughout this section, $k$ is a field with ${\rm char}\,k\neq 2$, unless otherwise specified.

When $a_i\in k^\times$, we will write the quadratic form
$q(X_1,\ldots,X_n)=\sum_{1\leq i\leq n}a_iX_i^2$
simply as $q=\langle a_1,\ldots,a_n\rangle$.
For brevity the quadratic form $q=\langle 1,1,1,-7\rangle$
will be written as $3\langle 1\rangle\bot\langle -7\rangle$.
A quadratic form $q=\langle a_1,\ldots,a_n\rangle$ is called
{\it isotropic over $k$} if $\sum_{1\leq i\leq n}a_ix_i^2=0$ for some $x_1,\ldots,x_n\in k$,
not all zero; otherwise, it will be called {\it anisotropic over $k$}.

An element $\alpha\in k^\times$ is
{\it a sum of $d$ squares in $k$}
if $\alpha=\sum_{1\leq i\leq d}x_i^2$ for some $x_i\in k$;
we will indicate this case by writing $\alpha=d\,\square^2$ in $k$ as a shorthand.

As mentioned in Section \ref{s1}, we will find a sufficient condition
for an algebraic number field $k$ such that $7\neq 3\,\square^2$ in $k$, but
$7=3\,\square^2$ in $k(\sqrt{2})$.

Note that $7=1^2+(\sqrt{2})^2+2^2$ in $\bQ(\sqrt{2})$.
Hence, for any field $k$ containing $\bQ$, $7=3\,\square^2$ in $k(\sqrt{2})$. The following lemma, attributed to Gauss, is well-known.

\begin{lemma}[{\cite[pages 173--174, Example 3.12]{La}, \cite[page 45, Appendix]{Se1}}]\label{l4.1}
If $n$ is a positive integer of the form
$7+8m$ for some integer $m\geq 0$, then $n\neq 3\,\square^2$ in $\bQ$.
\end{lemma}

\begin{lemma}\label{l4.2}
Assume that $7\neq 3\,\square^2$ in a field $k$.
Let $k^\prime$ be a field extension of $k$. Then\\
{\rm (1)} If $[k^\prime:k]$ is an odd integer, then $7\neq 3\,\square^2$
in $k^\prime$.\\
{\rm (2)} If $k\subset k^\prime\subset k(X_1,\ldots,X_n)$
where $k(X_1,\ldots,X_n)$ is the rational function field over $k$,
then $7\neq 3\,\square^2$ in $k^\prime$.
\end{lemma}

\begin{proof}
Note that $3\langle 1\rangle$ is anisotropic.
Otherewise, it contains a hyperbolic plane $\langle 1,-1\rangle$,
i.e. $3\langle 1\rangle\simeq\langle 1,-1\rangle\bot\langle -1\rangle$
by \cite[page 13, Theorem 3.4]{La}.
Thus $7=2\,\square^2$ in $k$, because $7$ is represented by
$\langle 1,-1\rangle$.
It follows that $7\neq 3\,\square^2$ in $k$
if and only if $3\langle 1\rangle\bot\langle -7\rangle$
is anisotropic over $k$.

For (1), apply Springer's theorem \cite[page 198, Theorem 2.3]{La}.

For (2), apply \cite[page 256, Corollary 1.2]{La}.
\end{proof}

\begin{lemma}[{Pfister \cite[Satz 5]{Pf}}]\label{l4.3}
Assume that $-1=s\,\square^2$ in a field $k$ where $s$ is a positive integer.
If $2^d\leq s< 2^{d+1}$ for some non-negative integer $d$, then
$-1=2^d\,\square^2$ in the field $k$.
\end{lemma}
\begin{proof}
Let $s(k)$ be the level of the field $k$, i.e.
$s(k):={\rm min}\{n\in\bN:-1=n\,\square^2\ \textrm{in}\ k\}$
and $s(k)=\infty$ if $n\langle 1\rangle$ is anisotropic for all $n\in\bN$.
Since $-1=s\,\square^2$ in $k$ by the assumption,
by Pfister's theorem \cite[page 303, Theorem 2.2]{La},
we find that $s(k)\leq 2^d$.
\end{proof}

\begin{lemma}\label{l4.4}
Let $k=\bQ(\sqrt{d})$ be the quadratic field
such that $d\equiv 1\pmod{8}$ where $d$ is a square-free integer
and $d\neq 1$.
Then $7\neq 3\,\square^2$ in $k$.
\end{lemma}
\begin{proof}
By \cite[page 313, Theorem 32]{ZS},
the prime $2$ splits completely in $k$,
i.e. $2\mathfrak{O}_k=\mathfrak{p}_1\mathfrak{p}_2$
where $\mathfrak{O}_k$ is the ring of integers of $k$
and $\mathfrak{p}_1$ and $\mathfrak{p}_2$ are distinct prime idelas of
$\mathfrak{O}_k$.
Thus $e_i=f_i=1$ for $i=1,2$
where $e_i$, $f_i$ denote the ramification index and the residue class
degree of $\mathfrak{p}_i$ (see \cite[page 285, Theorem 21]{ZS}).
Let $k_i$ be the local field associated to $\mathfrak{p}_i$,
i.e. $k_i$ is the quotient field of the completion
associated to the DVR determined by $\mathfrak{p}_i$.
Since $e_i=f_i=1$, it follows that the embedding
$\bQ_2\rightarrow k_i$ is surjective (i.e. an isomorphism)
by \cite[page 36, Section 5]{Se2}.

We will show that $7\neq 3\,\square^2$ in $k$.

Suppose not.
Then $3\langle 1\rangle\bot\langle -7\rangle$ is isotropic over $k$.
Hence $3\langle 1\rangle\bot\langle -7\rangle$ is also isotropic over $k_i$
for $i=1,2$.
Since $k_i\simeq\bQ_2$, it follows that
$3\langle 1\rangle\bot\langle -7\rangle$ is isotropic over $\bQ_2$.

On the other hand, $3\langle 1\rangle$ is isotropic over $\bQ_p$
for any odd prime number $p$ by \cite[page 150, Corollary 2.5]{La}.

Apply the Hasse-Minkowski theorem \cite[page 168, Theorem 3.1]{La}
to the indefinite form $3\langle 1\rangle\bot\langle -7\rangle$
over $\bQ$.
We find that $3\langle 1\rangle\bot\langle -7\rangle$
is isotropic over $\bQ$.
But this is impossible by Lemma \ref{l4.1}.
\end{proof}

\begin{remark}
If $\alpha=7+8m(m+1)\in\bN$ where $m$ is a non-negative integer,
the above proof remains valid for $\alpha$, i.e. $\alpha\neq 3\,\square^2$
in $k$, but $\alpha=1+((2m+1)\sqrt{2})^2+2^2$ in $k(\sqrt{2})$,
because $7+8d\neq 3\,\square^2$ in $\bQ$ for any non-negative integer $d$ by Lemma \ref{l4.1}.
\end{remark}

Once Lemma \ref{l4.4} is available, the proof of Theorem \ref{t1.5}
is rather obvious to those who are versed in Serre's ideas of
the proofs for \cite[page 85, Proposition 32.25]{GMS},
\cite[page 89, Theorem 33.26]{GMS},
\cite[pages 90--92, Section 34]{GMS}.

For the convenience of the reader, we choose to rewrite some key points
of these proofs at a leisurely pace so that the reader may accept
the proof comfortably.

As before, $k$ is a field of ${\rm char}\,k\neq 2$,
and $K$ is a field extension of $k$.
In Section \ref{s2}, the set $H^1(K,G)$ denotes the set of all the
(continuous) group homomorphisms from $\Gamma_K$ to the finite group $G$;
the set $H^1(K,G)$ is bijective to the set of $G$-Galois algebras
over $K$.
When $G=S_n$, the symmetric group of degree $n$,
there exists another bijection of $H^1(K,S_n)$ to
${\rm Et}_n(K)$, the set of all the \'etale algebras $E$ over $K$ with $[E:K]=n$
by the Galois descent (see \cite[page 153, Proposition 4]{Se2},
\cite[page 9, Example 3.2]{GMS}).
For each morphism $\varphi:\Gamma_K\to S_n$,
there corresponds a unique $n$-\'etale algebra $E_\varphi$ over $K$;
if $\varphi$ is understood from the context we may write $E$ for $E_\varphi$.

If $E\in {\rm Et}_n(K)$, we may define the trace form of $E$ over $K$,
denoted by $q_E$, as $q_E(x)={\rm Tr}_{E/K}(x^2)$
(see \cite[page 59, 25.7]{GMS}, \cite{Se3}).
On the other hand, if $q=\langle a_1,\ldots,a_n\rangle$ is a quadratic form
over $K$ (where $a_1,\ldots,a_n\in K^\times)$, we define the
Stiefel-Whitney classes $w_1(q),\ldots$ as follows (see \cite[page 41]{GMS})
\begin{align*}
w_1(q)&=\delta(a_1\cdots a_n)\in H^1(K,\bZ/2\bZ)\simeq K^\times/K^{\times 2},\\
w_2(q)&=\sum_{1\leq i<j\leq n}\delta(a_i)\cup\delta(a_j)\in H^2(K,\bZ/2\bZ)\\
\cdots.
\end{align*}

For simplicity, we will write $\delta(a_i)\cup\delta(a_j)$ as $(a_i)\cup (a_j)\in H^2(K,\bZ/2\bZ)$; similarly, we write $w_1(q)$ as $(a_1\cdots a_n)\in H^1(K,\bZ/2\bZ)$.

If $E\in {\rm Et}_n(K)$, note that $w_1(q_E)$ and $w_2(q_E)$ are well-defined.

\begin{lemma}[{\cite[page 662, 3.1]{Se3}, \cite[page 99, Corollary 9.2.3]{Se4}, \cite[page 88, 33.18]{GMS}}]\label{l4.5}
For $n\geq 4$, let $\varphi:\Gamma_K\to S_n$ be a continuous morphism and
$E\in {\rm Et}_n(K)$ be the \'etale algebra corresponding to $\varphi$.
Let $\widetilde{A_n}\xrightarrow{\varepsilon} A_n$ be the non-split double covering
of the alternating group $A_n$; write $1\to \{\pm 1\}\to \widetilde{A_n}\xrightarrow{\varepsilon} A_n \to 1$ for the central extension associated to $\widetilde{A_n}\xrightarrow{\varepsilon} A_n$. Let $\iota:A_n\to S_n$ be the inclusion map. Then\\
{\rm (1)} $\varphi(\Gamma_K)\leq A_n$ if and only if $w_1(q_E)=0$.\\
{\rm (2)} In case $\varphi(\Gamma_K)\leq A_n$, there exists a morphism $\varphi_0:\Gamma_K\to\widetilde{A_n}$ such that $\varphi=\iota \circ \varepsilon\circ\varphi_0$
if and only if $w_2(q_E)=0$.
\end{lemma}

\begin{lemma}[{\cite[page 216, Corollary 8.20.10]{Be}}]\label{l4.6}
Let $K$ be a field with ${\rm char}\,K\neq 2$,
${\rm Br}(K)$ be the  Brauer group of $K$,
${}_2{\rm Br}(K):=\{[A]\in {\rm Br}(K):2[A]=0\}$
and
$\left(\frac{a,b}{K}\right)$ be the quoternion algebra over $K$
where $a,b\in K^\times$ $($see \cite{Dr}$)$.
There exists a natural isomorphism $\Phi:{}_2{\rm Br}(K)\to H^2(K,\bZ/2\bZ)$
which sends the similarity class of $\left(\frac{a,b}{K}\right)$ to
the cup product $(a)\cup (b)$.
\end{lemma}

\begin{lemma}[{\cite [page 155, Theorem 5.4]{Sc}}]\label{l4.7}
Let $k$ be a field with ${\rm char}\,k\neq 2$.
Let $\psi$ and $\varphi$ be two quadratic forms over $k$
satisfying that
{\rm (i)} $\psi$ is anisotropic over $k$ with ${\rm dim}\,\psi=2m$
for some $m$ and {\rm (ii)} $\varphi=\langle a_1,\ldots,a_n\rangle$
where $a_i\in k^\times$ and $a_1=1$.
Define $\widehat{K}$ to be the quotient field of the integral domain
$k[X_1,\ldots,X_n]/\langle a_1X_1^2+\cdots+a_nX_n^2\rangle$ where
$k[X_1,\ldots,X_n]$ is the polynomial ring of $n$ variables over $k$.
If $\psi\otimes_k \widehat{K}\simeq m\langle 1\rangle\bot m\langle -1\rangle$
over $\widehat{K}$, then $2m\geq n$ and
$\psi\simeq \varphi\bot\varphi^\prime$ over $k$
where $\varphi^\prime$ is some quadratic form over $k$.
\end{lemma}

\begin{definition}\label{d4.8}
We will define an invariant $e\in {\rm Inv}^3_{{\rm nr},k}(\widetilde{A_n},\bZ/2\bZ)_{\rm norm}$ for $n=6,7$.
For any field extension $k^\prime$ of $k$,
for any continuous morphism $\varphi:\Gamma_{k^\prime}\to\widetilde{A_6}$,
let $\widehat{\varphi}$ be the composite of $\Gamma_{k^\prime}\xrightarrow{\varphi}\widetilde{A_6}\xrightarrow{\varepsilon}A_6\xrightarrow{\iota} S_6$
where $\varepsilon:\widetilde{A_6}\to A_6$ and $\iota:A_6\to S_6$ are the same as
in Lemma \ref{l4.5}.

Let $E/k^\prime$ be the \'etale algebra corresponding to
$\widehat{\varphi}:\Gamma_{k^\prime}\to S_6$ and $q_E$ be the trace form
of $E$ over $k^\prime$ (note that $[E:k^\prime]=6)$.
By Lemma \ref{l4.5}, $w_1(q_E)=w_2(q_E)=0$.
Apply \cite[page 83, Theorem 32.10]{GMS}.
We find that $q_E\simeq 2\langle 1\rangle\bot 4\langle c\rangle$
for some $c\in k^\times$.
Moreover, $(2)\cup (c)=0$ and $c=4\,\square^2$ in $k^\prime(\sqrt{2})$.
Now we define
\begin{align*}
e_{k^\prime}:H^1(k^\prime,\widetilde{A_6})&\longrightarrow H^3(k^\prime,\bZ/2\bZ)\\
\varphi\ &\longmapsto (c)\cup (-1)\cup (-1).
\end{align*}

If $\varphi:\Gamma_{k^\prime}\to\widetilde{A_6}$ is the trivial map,
i.e. $\varphi(\sigma)=1$ for all $\sigma\in\Gamma_{k^\prime}$,
then $\widehat{\varphi}:\Gamma_{k^\prime}\to S_6$ is also the trivial map
and the corresponding \'etale algebra is the split algebra $E^{\rm split}$
(see \cite[page 9, Example 3.2]{GMS}); thus the trace form is
$6\langle 1\rangle$.
Hence $e_{k^\prime}(\varphi)=0$ in $H^3(k^\prime,\bZ/2\bZ)$.

On the other hand, since $(2)\cup (c)=0$ and $c=4\,\square^2$ in
$k^\prime(\sqrt{2})$, we may apply \cite[page 84, Lemma 32.18]{GMS}.
It follows that $e_{k^\prime}(\varphi)\in H^3(k^\prime,\bZ/2\bZ)$
is a unramified cohomological class.

Conclusion.
$e$ is normalized and unramified.
For details, see \cite[page 89]{GMS}.

\medskip
Now we turn to $\widetilde{A_7}$.

Let $\varphi,\tilde{\varphi}, E$ be the same as before
(but $n=6$ is replaced by $n=7$).
By \cite[page 79, Corollary 31.24]{GMS},
$q_E\simeq\langle 1\rangle\bot q^\prime$ where
$q^\prime$ is a trace form of rank $6$.
Since $w_1(q_E)\simeq w_2(q_E)=0$,
it follows that $w_1(q^\prime)=w_2(q^\prime)=0$.
Then we find that $q^\prime=2\langle 1\rangle\bot 4\langle c\rangle$
as before.
We define
\begin{align*}
e_{k^\prime}:H^1(k^\prime,\widetilde{A_7})&\longrightarrow H^3(k^\prime,\bZ/2\bZ)\\
\varphi\ &\longmapsto (c)\cup (-1)\cup (-1).
\end{align*}
Clearly, $e\in {\rm Inv}^3_{{\rm nr},k}(\widetilde{A_7},\bZ/2\bZ)_{\rm norm}$ because we use the same cohoomlogical class in $\widetilde{A_7}$ and also in $\widetilde{A_6}$.
\end{definition}

\begin{theorem}\label{t4.9}
Let $e$ be the invariant defined in {\rm Definition \ref{d4.8}}.
If $k$ is a field containing $\bQ$ such that the quadratic forms
$3\langle 1\rangle\bot\langle -7\rangle$ and
$8\langle 1\rangle$ are anisotropic over $k$,
then $e$ is a non-zero invariant.
\end{theorem}
\begin{proof}
We will follow the proof of \cite[page 85, Proposition 32.25]{GMS}.
It suffices to show that the invariant $e$ is not zero where $e\in {\rm Inv}^3_{{\rm nr},k}(\widetilde{A_6},\bZ/2\bZ)_{\rm norm}$. The proof for the case of $\widetilde{A_7}$ follows from the case of $\widetilde{A_6}$.

\bigskip
Step 1.
Define $K=k(x_1,x_2,x_3,x_4,x_5)$ with relation
$x_1^2+x_2^2+x_3^2+x_4^2+x_5^2+7=0$.

We will show that
\begin{align*}
e_K:H^1(K,\widetilde{A_6})\longrightarrow H^3(K,\bZ/2\bZ)
\end{align*}
is a non-zero map.

The crux is to consider $q=2\langle 1\rangle\bot 4\langle -1\rangle$, a quadratic form
over $K$ and to show that $q$ is the trace form of some \'etale algebra
$E/K$ corresponding to the morphism
$\Gamma_K\xrightarrow{\varphi}\widetilde{A_6}\xrightarrow{\varepsilon}A_6\xrightarrow{\iota} S_6$ where $\varphi\in H^1(K,\widetilde{A_6})$ is some morphism, and then to show that $e_K(\varphi)=(-1)\cup (-1)\cup (-1)$ is a non-zero cohomology class.

\bigskip
Step 2.
Since $2\cdot H^i(K,\bZ/2\bZ)=0$ ($i=1$ or $2$), it is clear that $w_1(q)=w_2(q)=0$.
Thus we may use \cite[page 83, Theorem 32.10]{GMS} to ensure that
$q=2\langle 1\rangle\bot 4\langle -1\rangle$ is a trace form over $K$.
That is, for $c=-1$, we should check whether the three conditions
(i) $(2)\cup (c)=0$, (ii) $q\simeq 2\langle 1\rangle\bot 4\langle c\rangle$,
(iii) $c=4\,\square^2$ in $K(\sqrt{2})$
in \cite[page 83, Theorem 32.10]{GMS} are valid.

For (ii), $q=2\langle 1\rangle\bot 4\langle c\rangle$ with $c=-1$.

For (i), $(2)\cup (c)=(2)\cup (-1)$.
But $(2)\cup (-1)$ is isomorphic to the quoternion algebra
$\left(\frac{2,-1}{K}\right)$ by Lemma \ref{l4.6}.
Now $2={\rm Norm}_{K(\sqrt{-1})/K}(1+\sqrt{-1})$.
Hence $\left(\frac{2,-1}{K}\right)$ is isomorphic to the matrix ring over $K$, i.e. it belongs to
the trivial similarity class in ${\rm Br}(K)$.
Thus $(2)\cup (-1)=0$.

For (iii), we first prove that
$-1=4\,\square^2$ in $K(\sqrt{2})$ if and only if
$8\langle1\rangle$ is isotropic over $K(\sqrt{2})$.

The direction ``$\Rightarrow$'' is trivial.
Consider the direction ``$\Leftarrow$''.

If $8\langle 1\rangle$ is isotropic over $K(\sqrt{2})$,
then $-1=7\,\square^2$ in $K(\sqrt{2})$.
By Lemma \ref{l4.3}, $-1=4\,\square^2$ in $K(\sqrt{2})$.

Hence to verify the condition (iii),
it is equivalent to show that
$8\langle 1\rangle$ is isotropic over $K(\sqrt{2})$.

From $7=1+(\sqrt{2})^2+2^2$,
we find that $x_1^2+x_2^2+x_3^2+x_4^2+x_5^2+1^2+(\sqrt{2})^2+2^2=0$
by the definition of the field $K$.
Thus $8\langle 1\rangle$ is isotropic over $K(\sqrt{2})$. Done.

Conclusion. $q=2\langle 1\rangle\bot 4\langle -1\rangle$
is a trace form over $K$.
Since $w_1(q)=w_2(q)=0$ (as shown before),
the corresponding morphism $\Gamma_K\to S_6$ arises
from $\Gamma_K\xrightarrow{\varphi}\widetilde{A_6}\xrightarrow{\varepsilon}A_6\xrightarrow{\iota} S_6$ where $\varphi\in H^1(K,\widetilde{A_6})$.
In short, there exists some $\varphi\in H^1(K,\widetilde{A_6})$
such that $e_K(\varphi)=(-1)\cup (-1)\cup (-1)$.

\bigskip
Step 3.
We will show that $e_K(\varphi)\neq 0$.

By \cite[page 83, Remark 32.13]{GMS},
$e_K(\varphi)\neq 0$ if and only if $-1\neq 4\,\square^2$ in $K$.

Suppose not.
We have $-1=4\,\square^2$ in $K$, i.e.
$5\langle 1\rangle$ is isotropic over $K$.
It follows that $8\langle 1\rangle$ is also isotropic over $K$.
Since $8\langle 1\rangle=\langle\langle 1,1,1\rangle\rangle$
is a 3-fold Pfister form, we find that $8\langle 1\rangle\simeq 4\langle 1\rangle\bot4\langle -1\rangle$ over $K$
because an isotropic Pfister form is hyperbolic
\cite[page 279, Corollary 1.6]{La}.
In summary, if $K^\prime$ is any field containing $K$,
then $8\langle 1\rangle\simeq 4\langle 1\rangle\bot4\langle -1\rangle$
over $K^\prime$.

\bigskip
Step 4.
Consider the quadratic forms $\psi=8\langle 1\rangle$,
$\varphi=5\langle 1\rangle\bot\langle 7\rangle$ over the field $K$.
Define $\widehat{K}$ to be the quotient field of the integral domain
$k[X_1,X_2,X_3,X_4,X_5,X_6]/\langle\sum_{1\leq i\leq 5}X_i^2+7X_6^2\rangle$.

Note that the field $K$ may be embedded into $\widehat{K}$ by sending
$x_i$ to $X_i/X_6$ for $1\leq i\leq 5$.

By Step 3, we find that $8\langle1\rangle\simeq 4\langle 1\rangle\bot4\langle -1\rangle$ over $\widehat{K}$.

We will apply Lemma \ref{l4.7} (note that $\psi=8\langle 1\rangle$
is anisotropic over $k$ by assumption).
It follows that $8\langle 1\rangle\simeq 5\langle 1\rangle\bot\langle 7\rangle\bot\varphi^\prime$ over $k$.
By Witt cancellation \cite[page 15, Theorem 4.2]{La},
we have $3\langle 1\rangle\simeq\langle 7\rangle\bot\varphi^\prime$.

Consequently, $7=3\,\square^2$ in $k$.
But we assume already that $3\langle 1\rangle\bot\langle -7\rangle$
is anisotropic over $k$.
Thus this leads to a contradiction.
In other words, $-1=4\,\square^2$ in $K$ is impossible.
\end{proof}

\begin{remark}
In the above proof, we have shown that : Let $K/k$ be the field extension defined in Step 1 where $k$ is a field satisfying that $8\langle 1\rangle$ is anisotropic over $k$. If $(-1)\cup (-1)\cup (-1)=0$ in $H^3(K, \bZ/2\bZ)$, then $3\langle 1\rangle\bot\langle -7\rangle$ is isotropic over $k$.
\end{remark}

{\it Proof of Theorem \ref{t1.5}.}

Step 1.
$G$ is a finite group whose $2$-Sylow subgroup is isomorphic to $Q_{16}$.
We check the proof in \cite[pages 90--92, Section 34]{GMS} (in particular, the proof of \cite[page 92, Proposition 34.6]{GMS}). It is proved there: If $k$ is a field containing $\bQ$ such that
${\rm Inv}^3_{{\rm nr},k}(\widetilde{A_6},\bZ/2\bZ)_{\rm norm}\neq 0$,
then ${\rm Inv}^3_{{\rm nr},k}(G,\bZ/2\bZ)_{\rm norm}\neq 0$. Thus we may apply Theorem \ref{t4.9} and Theorem \ref{t2.3}. It follws that $k(G)$ is not retract rational over $k$.

\medskip
Step 2. When $k=\bQ(\sqrt{d})$ where $d$ is a square-free positive integer
of the form $1+8m$ with $m\geq 2$,
then the quadratic form $3\langle 1\rangle\bot\langle -7\rangle$ is
anisotropic over $k$ by Lemma \ref{l4.4}.
The quadratic form $8\langle 1\rangle$ is also anisotropic over $k$
because $k$ is a subfield of the ordered field $\bR$. Apply Theorem \ref{t4.9}.

\medskip
Step 3. This step is unnecessary for the proof of Theorem \ref{t1.5}. We insert it here because we intend to call the reader's attention to the analogous situations of \cite[page 89]{GMS} and \cite[page 90--92]{GMS}.

Suppose that $G$ is a subgroup of odd index of $\widetilde{A_7}$ as in \cite[page 89, Theorem 32.26]{GMS}. Re-examine the proof therein.

It is easy to find that what is proved there is:
If $k$ is a field containing $\bQ$ such that
${\rm Inv}^3_{{\rm nr},k}(\widetilde{A_7},\bZ/2\bZ)_{\rm norm}\neq 0$,
then ${\rm Inv}^3_{{\rm nr},k}(G,\bZ/2\bZ)_{\rm norm}\neq 0$ for any such a group $G$.

Note that the above fact may be deduced also from the restriction-corestrection
arguments \cite[page 34, Proposition 14.4]{GMS}.

Use Theorem \ref{t4.9} and then apply Theorem \ref{t2.3} to deduce that
$k(G)$ is not retract rational over $k$.
\qed

\begin{remark}
(1) If $k=\bQ(\sqrt{-d})$ where $d$ is a square-free integer of the form
$7+8m$ for some integer $m\geq 0$, then $d\neq 3\,\square^2$ in $\bQ$
by Lemma \ref{l4.1}, and $d=4\,\square^2$ in $\bQ$ by
Legrange's four-square theorem.
Apply \cite[page 305, Theorem 2.7]{La}.
We find that the level of $k$ is $4$,
i.e. $-1=4\,\square^2$ in $k$.
Thus $8\langle 1\rangle$ is not anisotropic over $k$.

(2) If $k$ is a field containing $\bQ$ such that at least one of $-1, 2, -2$ is a square in $k$ (that is, $k(\zeta_8)/k$ is a cyclic extension), it is easy to see that $3\langle 1\rangle\bot\langle -7\rangle$ is isotropic over $k$.

(3) If $k\subset k^\prime$ are fields such that
{\rm (i)} $3\langle 1\rangle\bot\langle -7\rangle$ and
$8\langle 1\rangle$ are anisotropic over $k$,
and
{\rm (ii)} $k^\prime$ is an extension of $k$ by successive
extensions of odd-degree extensions and/or unirational extensions
(see Lemma \ref{l4.2}), then $3\langle 1\rangle\bot\langle -7\rangle$
and $8\langle 1\rangle$ are anisotropic over $k^\prime$.
\end{remark}


\end{document}